\documentclass[a4paper]{amsart}

\usepackage{amsrefs}
\usepackage{hyperref}

\usepackage{tikz-cd}

\theoremstyle{definition}
\newtheorem{definition}{Definition}

\theoremstyle{plain}
\newtheorem{theorem}[definition]{Theorem}
\newtheorem{proposition}[definition]{Proposition}
\newtheorem{lemma}[definition]{Lemma}

\theoremstyle{remark}

\newtheorem{example}[definition]{Example}
\newtheorem*{acknowledgement}{Acknowledgement}

\title{Yamabe Invariants and the $\Pin^-(2)$-monopole Equations}

\author[M. Ishida]{Masashi Ishida}
\address{Department of Mathematics, Graduate School of Science,
Osaka University, 1-1, Machikaneyama, Toyonaka, Osaka, 560-0043, Japan}
\email{ishida@math.sci.osaka-u.ac.jp }

\author[S. Matsuo]{Shinichiroh Matsuo}
\address{Department of Mathematics, Graduate School of Science,
Osaka University, 1-1, Machikaneyama, Toyonaka, Osaka, 560-0043, Japan}
\email{matsuo@math.sci.osaka-u.ac.jp }

\author[N. Nakamura]{Nobuhiro Nakamura}
\address{Osaka Medical College, 2-7 Daigaku-machi, Takatsuki City, Osaka, 569-8686, Japan}
\email{nobuhironakamura00@gmail.com}


\newcommand{\Z}{\mathbb{Z}}
\newcommand{\R}{\mathbb{R}}
\newcommand{\U}{\mathrm{U}}
\newcommand{\Sp}{\mathrm{Sp}}

\newcommand{\Pin}{\mathrm{Pin}}
\newcommand{\Spin}{\mathrm{Spin}}

\DeclareMathOperator{\rank}{rank}
\DeclareMathOperator{\Hom}{Hom}

\DeclareMathOperator{\SW}{SW}

\newcommand{\metrics}{\mathcal{M}}
\newcommand{\dv}{d\mu}

\newcommand{\EH}{E}

\DeclareMathOperator{\Ric}{Ric}
\newcommand{\tlRicci}{\overset{\circ}{r}}

\newcommand{\yinv}{\mathcal{Y}}
\newcommand{\infs}{\mathcal{I}_s}

\newcommand{\CP}{\mathbb{C}\mathrm{P}}

\begin{document}

\begin{abstract}
  We compute the Yamabe invariants for a new infinite class of closed $4$-dimensional manifolds by using a ``twisted'' version of the Seiberg-Witten equations, the $\Pin^-(2)$-monopole equations.
  The same technique also provides a new obstruction to the existence of Einstein metrics or long-time solutions of the normalised Ricci flow with uniformly bounded scalar curvature.
\end{abstract}
\maketitle

\section{Introduction}
  The Yamabe invariant is a diffeomorphism invariant of smooth manifolds, which arises from a variational problem for the total scalar curvature of Riemannian metrics.
  The $\Pin^-(2)$-monopole equations are a ``twisted'' version of the Seiberg-Witten equations.
  In this paper we will compute the Yamabe invariants for a new infinite class of closed $4$-dimensional manifolds by using the $\Pin^-(2)$-monopole equations.
  
  We begin by recalling the Yamabe invariant.
  Let $X$ be a closed, oriented, connected manifold of $\dim X = m \ge 3$, and $\metrics (X)$ the space of all smooth Riemannian metrics on $X$.
  For each metric $g \in \metrics (X)$, we denote by $s_g$ the scalar curvature and by $\dv_g$ the volume form.
  Then the normalised Einstein-Hilbert functional $\EH_X \colon \metrics (X) \to \R$ is defined by
  \[
	\EH_X \colon g \mapsto \frac{\int_X s_g \,\dv_g}{\left( \int_X \dv_g \right)^{\frac{m-2}{m}}}.
  \]
  The classical Yamabe problem is to find a metric $\check{g}$ in a given conformal class $C$ such that the normalised Einstein-Hilbert functional attains its minimum on $C$: $\EH_X (\check{g}) = \inf_{g \in C} \EH_X (g)$.
  This minimising metric $\check{g}$ is called a Yamabe metric, and a conformal invariant $\yinv (X, C) := \EH_X(\check{g})$ the Yamabe constant.
  We define a diffeomorphism invariant $\yinv (X)$ by the supremum of $\yinv (X, C)$ of all the conformal classes $C$ on $X$:
  \[
    \yinv (X) := \sup_{C} \yinv (X, C) = \sup_{C} \inf_{g} \frac{\int_X s_g \,\dv_g}{\left( \int_X \dv_g \right)^{\frac{m-2}{m}}}
  \]
  We call it the Yamabe invariant of $X$; it is also referred to as the $\sigma$-constant.
  See \cite{MR919505} and \cite{MR994021}.
  
  It is a natural problem to compute the Yamabe invariant.  
  In dimension $4$, Seiberg-Witten theory and LeBrun's curvature estimates have played a prominent role in this problem.
  LeBrun used the ordinary Seiberg-Witten equations to compute the Yamabe invariants of most algebraic surfaces~\cites{MR1386835, MR1674105}.
  In particular, he showed that a compact K\"ahler surface is of general type if and only if its Yamabe invariant is negative.
  He also showed $\yinv (\CP^2) = 12\sqrt{2}\pi$ via the perturbed Seiberg-Witten equations \cite{MR1487727}.
  Bauer and Furuta's stable cohomotopy Seiberg-Witten invariant \cite{MR2025298} or Sasahira's spin bordism Seiberg-Witten invariant \cite{MR2284407} enable us to compute the Yamabe invariants of connected sums of some compact K\"ahler surfaces \cites{MR2032500, MR2284407, IS-1, IS-2}.
  In this paper, we will employ a recently introduced ``twisted" version of the Seiberg-Witten invariant, the $\Pin^-(2)$-monopole invariant~\cite{N2}, to compute the Yamabe invariants for a new infinite class of $4$-dimensional manifolds.
  The advantage of using this new invariant lies in the fact that it can be non-trivial even when the ordinary Seiberg-Witten invariants, the spin bordism Seiberg-Witten invariants, and the stable cohomotopy Seiberg-Witten invariants all vanish.
  Example \ref{ex: Z} lies at the heart of this paper.
  
  We now state the main theorems of this paper.
  In what follows, $\chi(X)$ and $\tau(X)$ denote the Euler number and the signature of a manifold $X$ respectively, and $mX := X \# \cdots \# X$ denotes the $m$-fold connected sum.
  
  \begin{theorem} \label{thm:yamabe}
  Let $M$ be a compact, connected, minimal K\"{a}hler surface with $b_+(M)\ge 2$ and $c^{2}_{1}(M) = 2\chi(M)+3\tau(M)\ge 0$. 
  Let $N$ be a closed, oriented, connected $4$-manifold with $b_+(N)=0$ and $\yinv(N) \ge 0$.
  Let $Z$ be a connected sum of arbitrary positive number of $4$-manifolds, each of which belongs to one of the following types:
  \begin{enumerate}
    \item $S^2 \times \Sigma$, where $\Sigma$ is a compact Riemann surface with positive genus, or
    \item $S^1\times Y$, where $Y$ is a closed oriented $3$-manifold with $\yinv(Y) \ge 0$.
  \end{enumerate}  
  The Yamabe invariant of the connected sum $M\#N\#Z$ is equal to $-4\pi \sqrt{2c^{2}_{1}(M)}$.
  \end{theorem}
  
  \begin{theorem} \label{thm:Enriques0}
    Let $M$ be an Enriques surface.
    Let $N$ and $Z$ satisfy the assumptions in Theorem \ref{thm:yamabe}.
    The Yamabe invariant of $M \# N \# Z$ is equal to $0$.
  \end{theorem}
  
  The key ingredients of the proofs are Proposition \ref{prop:pin2-class a} and Proposition \ref{prop:pin2-class b}, the non-vanishing of the $\Pin^-(2)$-monopole invariants of $M \# N \# Z$.
  We emphasise that the ordinary Seiberg-Witten invariants, the spin bordism Seiberg-Witten invariants, and the stable cohomotopy Seiberg-Witten invariants all vanish if $Z$ contains at least one $S^2 \times \Sigma$ as a connected-summand.
  
  Much more subtle is the following theorem.
  In general, the moduli spaces of the $\Pin^-(2)$-monopole equations are, in contrast to ordinary Seiberg-Witten theory, not orientable, and only $\Z_2$-valued invariants are defined; these invariants are powerful enough to prove the theorems above.
  \begin{theorem} \label{thm:Enriques}
    Let $M$ be an Enriques surface.
    Let $N$ be a closed, oriented, connected $4$-manifold with $b_+(N)=0$ and $\yinv(N) \ge 0$.
    For any $m \ge 2$, the Yamabe invariant of $mM \# N$ is equal to $0$; moreover, it does not admit Riemannian metrics of non-negative scalar curvature.
  \end{theorem}
  The ordinary Seiberg-Witten invariants of $mM$ are trivial; furthermore, its $\Z_2$-valued $\Pin^-(2)$-monopole invariants are also trivial \cite{N2}*{Theorem 1.13}.
  We need refined $\Z$-valued $\Pin^-(2)$-monopole invariants to prove the last theorem.

\begin{acknowledgement}
  The authors gratefully acknowledge the many helpful suggestions of the anonymous referee.
  The first author is supported in part by Grant-in-Aid for Scientific Research (C) 25400074.
  The second author is supported in part by Grant-in-Aid for Young Scientists (B) 25800045.
  The third author is supported in part by Grant-in-Aid for Scientific Research (C) 25400096.
\end{acknowledgement}

\section{The $\Pin^-(2)$-monopole equations and LeBrun's curvature estimates}\label{sec:estimate}

\subsection{The $\Pin^-(2)$-monopole equations}
  We briefly review $\Pin^-(2)$-monopole theory; for a thorough treatment, we refer the reader to \cites{MR3118618,N2}.
  
  Let $X$ be a closed, oriented, connected $4$-manifold.
  Fix a Riemannian metric $g$ on $X$.
  Let $\widetilde{X} \to X$ be an unbranched double cover, and $\ell := \widetilde{X} \times_{\{\pm 1\}} \Z$ its associated local system.
  Let $b_j^{\ell}(X) := \rank H^j(X;\ell)$ and $b_+^{\ell}(X) := \rank H^+(X;\ell)$.
  Recall that $\Pin^-(2) := \U(1) \cup jU(1) \subset \Sp(1)$ and $\Spin^{c-}(4) := \Spin(4) \times_{\{\pm 1\}} \Pin^-(2)$.
    A $\Spin^{c-}$-structure on $\widetilde{X} \to X$ is defined to be a triple $\mathfrak{s} = (P, \sigma, \tau)$, where
    \begin{itemize}
      \item $P$ is a $\Spin^{c-}(4)$-bundle on $X$,
      \item $\sigma$ is an isomorphism between $\widetilde{X}$ and $P/\Spin^c(4)$, and
      \item $\tau$ is an isomorphism between the frame bundle of $X$ and $P/\Pin^-(2)$.
    \end{itemize}
  We call the associated $\mathrm{O}(2)$-bundle $E := P/\Spin(4)$ the characteristic bundle of a $\Spin^{c-}$-structure $\mathfrak{s} = (P,\sigma,\tau)$, and denote its $\ell$-coefficient Euler class by $\widetilde{c}_1(\mathfrak{s}) \in H^2(X; \ell)$.
  If $\widetilde{X} \to X$ is trivial, any $\Spin^{c-}$-structure on $\widetilde{X} \to X$ canonically induces a $\Spin^c$-structure on $X$ \cite{N2}*{2(iv)}.
 
  $\Spin^{c-}$-structures are in many ways like $\Spin^c$-structures:
  The $\Spin^{c-}$-structure $\mathfrak{s}$ on $\widetilde{X} \to X$ determines a triple $(S^+, S^-, \rho)$, where $S^{\pm}$ are the spinor bundles on $X$ and $\rho \colon \Omega^1(X; \ell \otimes \sqrt{-1}\R) \to \Hom(S^+,S^-)$ is the Clifford multiplication.
  An $\mathrm{O}(2)$-connection $A$ on $E$ gives a Dirac operator $D_A \colon \Gamma(S^+) \to \Gamma(S^-)$. 
  Note that $F_A^+ \in \Omega^+(X; \ell \otimes \sqrt{-1}\R)$.
  The canonical real quadratic map is denoted by $q \colon S^+ \to \Omega^+(X; \ell \otimes \sqrt{-1}\R)$.

  We denote by $\mathcal{A}$ the space of $\mathrm{O}(2)$-connections on $E$.
  Let $\mathcal{C} := \mathcal{A} \times \Gamma(S^+)$ and $\mathcal{C}^* := \mathcal{A} \times (\Gamma(S^+) \setminus \{0\})$. 
  We define the $\Pin^-(2)$-monopole equations to be
  \[
	\left\{
	\begin{aligned}
		D_A \Phi &= 0 \\
		\frac{1}{2} F_A^+ &= q(\Phi)
	\end{aligned}
	\right.
  \]
  for $(A,\Phi) \in \mathcal{C}$.
  The gauge group $\mathcal{G} := \Gamma(\widetilde{X} \times_{\{\pm 1\}} \U(1))$ acts on the set of solutions of these equations; the moduli space is defined to be the set of solutions modulo $\mathcal{G}$.
  The formal dimension of the moduli space is given by
  \[
    d(\mathfrak{s}) := \frac{1}{4}\big(\widetilde{c}_1(\mathfrak{s})^2 - \tau(X)\big) - \big(b_+^{\ell}(X) - b_1^{\ell}(X) + b_0^{\ell}(X)\big).
  \]
  Note that $b_0^{\ell}(X) = 0$ if $\widetilde{X}$ is non-trivial.
  
  Let $\mathcal{B}^* := \mathcal{C}^*/\mathcal{G}$ be the irreducible configuration space.
  As in ordinary Seiberg-Witten theory, we can define the $\Pin^-(2)$-monopole invariant
  \[
    \SW^{\Pin^-(2)}(X,\mathfrak{s}) \colon H^{d(\mathfrak{s})}(\mathcal{B}^*; \Z_2) \to \Z_2
  \]
  via intersection theory on the moduli space.
  In contrast to ordinary Seiberg-Witten theory, a moduli space of solutions of the $\Pin^-(2)$-monopole equations might not be orientable, and thus the invariant is, in general, $\Z_2$-valued.
  We remark, however, that, in the case of Theorem \ref{thm:Enriques}, the moduli spaces are orientable, and we will use the refined $\Z$-valued invariant \cite{N2}*{Theorem 1.13}.
  
  \begin{example} \label{ex: sigma}
    Let $\widetilde{T^2} \to T^2$ be a non-trivial double cover, and $\ell := \widetilde{T^2} \times_{\pm 1} \Z$ its associated local system.
    Set $\Sigma := T^2 \# \cdots \# T^2$.
    The connected sum $\ell \# \cdots \# \ell$ gives a local system on $\Sigma$.
    We define a local system $\ell_{\Sigma}$ on $S^2 \times \Sigma$ by the pull-back of $\ell \# \cdots \# \ell$ by the projection $S^2 \times \Sigma \to \Sigma$.
    Then, we have
    \begin{gather*}
      b_0^{\ell_{\Sigma}}(S^2 \times \Sigma) = b_2^{\ell_{\Sigma}}(S^2 \times \Sigma) = b_4^{\ell_{\Sigma}}(S^2 \times \Sigma) = 0 \\
      b_1^{\ell_{\Sigma}}(S^2 \times \Sigma) = b_3^{\ell_{\Sigma}}(S^2 \times \Sigma) = \chi(\Sigma).
    \end{gather*}
    In particular, $b_+^{\ell_{\Sigma}}(S^2 \times \Sigma) = 0$, while $b_+(S^2 \times \Sigma) > 0$.
  \end{example}
  \begin{example} \label{ex: s1}
    Let $Y$ be a closed oriented $3$-manifold.
    Let $\widetilde{S^1} \to S^1$ be a connected double cover.
    We define a non-trivial double cover $\widetilde{S^1 \times Y} \to S^1 \times Y$ by the pull-back of $\widetilde{S^1}$ by the projection $S^1 \times Y \to S^1$, and denote by $\ell_{S^1}$ its associated local system.
    Then, we have $b_j^{\ell_{S^1}}(S^1 \times Y) =  0$ for all $j = 0, \dots, 4$.
  \end{example}
  \begin{example} \label{ex: Z}
    Let $Z$ be a connected sum
    \[
      Z := \big( S^2 \times \Sigma_1 \# \dots \# S^2 \times \Sigma_n \big) \# \big( S^1 \times Y_1 \# \dots \# S^1 \times Y_m \big),
    \]
    where each $\Sigma_j$ is a Riemann surface of positive genus and each $Y_i$ is a closed oriented $3$-manifold.
    We define a non-trivial double cover $\widetilde{Z} \to Z$ by the connected sum of $\widetilde{S^2 \times \Sigma_j}$ and $\widetilde{S^1 \times Y_i}$ in Example \ref{ex: sigma} and \ref{ex: s1}, and denote by $\ell_Z$ its associated local system.
    We emphasise that $b_+^{\ell_Z}(Z) = 0$, even if $b_+(Z) > 0$.
    It follows that $\widetilde{c}_1(\mathfrak{s})$ is a torsion class for every $\Spin^{c-}$-structure $\mathfrak{s}$ on $\widetilde{Z} \to Z$. 
    See \cite{N2}*{Theorem 1.7}.
  \end{example}

\subsection{LeBrun's curvature estimates}
  \begin{definition}
    Let $X$ be a closed, oriented, connected $4$-manifold.
    Assume that $X$ has a non-trivial double cover $\widetilde{X} \to X$ with $b_+^{\ell}(X) \ge 2$, where $\ell := \widetilde{X} \times_{\pm 1} \Z$.
    A cohomology class $\mathfrak{a} \in H^2(X; \ell) / \mathrm{Tor}$ is called a $\Pin^-(2)$-basic class if there exists a $\Spin^{c-}$-structure $\mathfrak{s}$ on $\widetilde{X} \to X$ with $\widetilde{c}_1(\mathfrak{s}) = \mathfrak{a}$ modulo torsions for which the $\Pin^-(2)$-monopole invariant is non-trivial.    
  \end{definition}
  
  As in ordinary Seiberg-Witten theory, if $X$ has a $\Pin^-(2)$-basic class, the corresponding $\Pin^-(2)$-monopole equations have at least one solution for every Riemannian metric; hence, $X$ does not admit Riemannian metrics of positive scalar curvature.
  We have, moreover, LeBrun's curvature estimates, which we will explain.
  In what follows, given a Riemannian metric $g$ on $X$, we identify $H^2(X; \ell \otimes \R)$ with the space of $\ell$-coefficient $g$-harmonic $2$-forms, and denote by $\mathfrak{a}^{+_g}$ the $g$-self-dual part of $\mathfrak{a} \in H^2(X; \ell) / \mathrm{Tor} \subset H^2(X; \ell \otimes \R)$.

  \begin{proposition} \label{prop:scalar}
     Let $X$ be a closed, oriented, connected $4$-manifold.
    Assume that $X$ has a non-trivial double cover $\pi \colon \widetilde{X} \to X$ with $b_+^{\ell}(X) \ge 2$, where $\ell := \widetilde{X} \times_{\{\pm 1\}} \Z$.
    If there exists a $\Pin^-(2)$-basic class $\mathfrak{a} \in H^2(X; \ell) / \mathrm{Tor}$, then the following hold for every Riemannian metric $g$ on $X$:
    \begin{itemize}
      \item The scalar curvature $s_g$ of $g$ satisfies
        \begin{equation} \label{eq:scalar}
          \int_X s_g^2 \,\dv_g \ge 32\pi^2 (\mathfrak{a}^{+_g})^2.
        \end{equation}
      If $\mathfrak{a}^{+_g} \ne 0$, equality holds if and only if there exists an integrable complex structure on the double cover $\widetilde{X}$ compatible with the pulled-back metric $\widetilde{g} := \pi^*g$ such that the covering transformation $\iota \colon \widetilde{X} \to \widetilde{X}$ is anti-holomorphic and the compatible K\"ahler form $\widetilde{\omega}$ satisfies $\iota^* \widetilde{\omega} = -\widetilde{\omega}$.
      \item The scalar curvature $s_g$ and the self-dual Weyl curvature $W^+_g$ of $g$ satisfy
        \begin{equation} \label{eq:weyl}
          \int_X \big(s_g - \sqrt{6}|W^+_g| \big)^2 \,\dv_g \geq 72\pi^2(\mathfrak{a}^{+_g})^2.
        \end{equation}
      If $\mathfrak{a}^{+_g} \ne 0$, equality holds if and only if the pulled-back metric $\tilde{g} := \pi^*g$ on $\widetilde{X}$ is an almost-K\"ahler metric with almost-K\"{a}hler form $\tilde{\omega}$ such that $\iota^*\tilde{\omega} = -\tilde{\omega}$.
    \end{itemize}
  \end{proposition}
  \begin{proof}
    LeBrun's arguments \cites{MR1872548, MR2039991} or the perturbations introduced in \cite{FM} are easily adapted to prove (\ref{eq:scalar}) and (\ref{eq:weyl}) by using the Weitzenb\"ock formulae of the Dirac operator for $\Spin^{c-}$-spinors and the Hodge Laplacian for $\ell$-coefficient self-dual forms.
    
    Assume that equality holds.
    We lift a solution $(A, \Phi)$ of the $\Pin^-(2)$-monopole equations on $X$ to the double cover $\widetilde{X}$.
    The lifted $\Spin^{c-}$-structure on $\widetilde{X}$ canonically reduces to a $\Spin^c$-structure, and the lifted solution $(\widetilde{A}, \widetilde{\Phi})$ can be identified with a solution of the ordinary Seiberg-Witten equations on $\widetilde{X}$ that satisfies
    \[
      \int_{\widetilde{X}} s_{\widetilde{g}}^2 \,\dv_{\widetilde{g}} = 32\pi^2 ((\pi^* \mathfrak{a})^{+_{\widetilde{g}}})^2, \text{ or }
      \int_{\widetilde{X}} \big( s_{\widetilde{g}} - \sqrt{6}|W^+_{\widetilde{g}}| \big)^2 \,\dv_{\widetilde{g}} = 72\pi^2 ((\pi^* \mathfrak{a})^{+_{\widetilde{g}}})^2.
    \]
    If the former (resp. latter) equality holds, the $\widetilde{g}$-self-dual form $\widetilde{\omega} = \sqrt{2} i q(\widetilde{\Phi})/|q(\widetilde{\Phi})|$ is a K\"ahler (resp. almost-K\"ahler) form compatible with $\widetilde{g}$.
    See \cite{MR2494171}*{Proposition 3.2 and Proposition 3.8}.
    Since $q(\widetilde{\Phi}) = \pi^* q(\Phi)$ and $iq(\Phi) \in \Omega^2(X; \ell \otimes \R)$, we have $\iota^* \widetilde{\omega} = - \widetilde{\omega}$.
    In the former case, moreover, $\iota$ is anti-holomorphic because $\iota^* \widetilde{g} = \widetilde{g}$.
  \end{proof}
  
\section{Gluing formulae and $\Pin^-(2)$-basic classes}
Based on gluing formulae for the $\Pin^-(2)$-monopole invariant \cite{N2}, we will establish the existence of $\Pin^-(2)$-basic classes on some classes of closed $4$-manifolds.

\subsection{Irreducible $\U(1)$ and reducible $\Pin^-(2)$.}
  We first establish a non-vanishing result based on a gluing formula for irreducible $\U(1)$-monopoles and reducible $\Pin^-(2)$-monopoles \cite{N2}*{Theorem 3.8}.
  It will play a pivotal role in the proof of Theorem \ref{thm:yamabe}.
  
  \begin{proposition} \label{prop:pin2-class a}
    Let $M$ be a closed, oriented, connected $4$-manifold that satisfies the following:
    \begin{itemize}
      \item $b_+(M)\geq 2$, and
      \item there exists a $\Spin^c$-structure $\mathfrak{s}_M$ such that $c_1(\mathfrak{s}_M)^2=2\chi(M)+3\tau(M)$ and its ordinary Seiberg-Witten invariant is odd. 
    \end{itemize}
    Let $N$ be a closed, oriented, connected $4$-manifold with $b_+(N) = 0$.
    Let $Z$ be a connected sum of arbitrary positive number of $4$-manifolds, each of which belongs to one of the following types:
    \begin{enumerate}
      \item $S^2\times\Sigma$, where $\Sigma$ is a compact Riemann surface with positive genus,
      \item $S^1\times Y$, where $Y$ is a closed oriented $3$-manifold.
    \end{enumerate}
    Set $X := M \# N \# Z$.
    Then, there exists a non-trivial double cover $\widetilde{X} \to X$ and a $\Pin^-(2)$-basic class $\mathfrak{a} \in H^2(X; \ell_X)$, where $\ell_X := \widetilde{X} \times_{\{\pm 1\}} \Z$, such that
    \[
      (\mathfrak{a}^{+_g})^2 \ge 2\chi(M) + 3\tau(M)
    \]
    for any Riemannian metric $g$ on $X$.
  \end{proposition}
  
  \begin{proof}
    Set $X_1 := M \# N$ and $X_2 := Z$.
    We will apply \cite{N2}*{Theorem 3.8} to $X = X_1 \# X_2$ as follows.
    
    We can choose a set of non-trivial smooth loops $\gamma_1, \dots, \gamma_b$ in $N$ so that surgery along them produces a $4$-manifold $N'$ with $b_1(N')=0$ and $b_+(N')=0$.
    Conversely, we can find a set of homologically trivial embedded $2$-spheres in $N'$ so that surgery along them recovers $N$.
    We will identify $H^2(N;\Z)$ with $H^2(N';\Z)$.
    
    Set $X_1' := M \# N'$.
    Let $e_1, \dots, e_k$ be a set of generators for $H^2(N';\Z) / \mathrm{Tor}$ relative to which the intersection form is diagonal \cite{MR710056}.
    By Froyshov's generalised blow-up formula \cite{MR2465077}*{Corollary 14.1.1}, $X_1'$ has a $\Spin^c$-structure $\mathfrak{s}_1'$ such that
    \[
      c_1(\mathfrak{s}_1') = c_1(\mathfrak{s}_M) + (\pm e_1 + \dots + \pm  e_k), 
    \]
    its ordinary Seiberg-Witten moduli space is $0$-dimensional, and its ordinary Seiberg-Witten invariant is equal to that of $(M,\mathfrak{s}_M)$.
    Here, the signs of $\pm e_i$ are arbitrary and independent of one another.
    
    By Ozsv\'ath and Szab\'o's surgery formula \cite{MR1863729}*{Proposition 2.2}, $X_1$ has a $\Spin^c$-structure $\mathfrak{s}_1$ such that $c_1(\mathfrak{s}_1) = c_1(\mathfrak{s}_1')$ and
    \[
      \SW^{\U(1)}(X_1, \mathfrak{s}_1) (\mu(\gamma_1) \cdots \mu(\gamma_b)) = \SW^{\U(1)}(X_1', \mathfrak{s}_1')(1)
    \]
    for some homology orientation on $X_1'$, where $ \SW^{\U(1)}$ denotes the ordinary Seiberg-Witten invariant and $\mu \colon H_1(X_1; \Z) \to H^1(\mathcal{B}^*; \Z)$ is a ``$\mu$-map'' to the irreducible configuration space $\mathcal{B}^* = \mathcal{B}^*(\mathfrak{s}_1)$.
    
    We take a non-trivial double cover $\widetilde{X_2} \to X_2$ as described in Example \ref{ex: Z}, and choose any $\Spin^{c-}$-structure $\mathfrak{s}_2$ on $\widetilde{X_2} \to X_2$.
    Note that $\widetilde{c}_1(\mathfrak{s}_2)^2=0$.
  
    Set $\widetilde{X} := X_1 \# X_1 \# \widetilde{X_2}$.
    It now follows from \cite{N2}*{Theorem 3.8} that
    \[
      \mathfrak{a} := c_1(\mathfrak{s}_M) + (\pm e_1 + \dots + \pm e_k) + \widetilde{c}_1(\mathfrak{s}_2)
    \]
    is a $\Pin^-(2)$-basic class.
    Given a Riemannian metric $g$ on $X$, we can choose the signs of $\pm e_i$ so that
    \[
      (\mathfrak{a}^{+_g})^2 \ge (c_1(\mathfrak{s}_M) + \widetilde{c}_1(\mathfrak{s}_2))^2 = 2\chi(M) + 3\tau(M)
    \]
    holds \cite{MR2032500}*{Corollary 11}.
    This completes the proof.
  \end{proof}
  
\subsection{Surgery formulae for the $\Pin^-(2)$-monopole invariant}
  We digress to generalise Ozsv\'ath and Szab\'o's surgery formula to the $\Pin^-(2)$-monopole invariant.
  
  We first describe a surgery formula for the $\Z_2$-valued $\Pin^-(2)$-monopole invariant,  which will be used to prove Proposition \ref{prop:pin2-class b}.
  Let $X$ be a closed, oriented, connected $4$-manifold and $\pi \colon \widetilde{X} \to X$ a non-trivial double cover.
  Fix a $\Spin^{c-}$-structure $\mathfrak{s}$ on $\widetilde{X} \to X$.
  Let $S \subset X$ be an embedded $2$-sphere with zero self-intersection number.
  Note that the restriction of $\mathfrak{s}$ to a tubular neighbourhood of $S$ is untwisted; therefore, it canonically induces a usual $\Spin^c$-structure on the neighbourhood.
  We denote by $X'$ the manifold obtained by surgery on $S$, and let $C \subset X'$ be the core of the added $S^1 \times D^3$.
  The inverse image $\pi^{-1}(S) \subset \widetilde{X}$ consists of disjoint embedded $2$-spheres $S_1$ and $S_2$.
  Equivariant surgery on $S_1$ and $S_2$ produces a double covering $\widetilde{X'} \to X'$.
  Let $\{C_1, C_2\} := \pi^{-1}(C) \subset \widetilde{X'}$.
  \[
  \begin{tikzcd}
    \widetilde{X} \setminus \{S_1 \cup S_2\} \arrow{d} \arrow{r}{\cong} & \widetilde{X'} \setminus \{C_1 \cup C_2\} \arrow{d} \\
    X \setminus S \arrow{r}{\cong} & X' \setminus C.
  \end{tikzcd}
  \]
  There is a unique $\Spin^{c-}$-structure $\mathfrak{s'}$ on $\widetilde{X'} \to X'$ with the property that
  \[
    \left. \mathfrak{s'} \right|_{\widetilde{X'} \setminus \{C_1 \cup C_2\} \to X \setminus S} = \left. \mathfrak{s} \right|_{\widetilde{X} \setminus \{S_1 \cup S_2\} \to X' \setminus C}.
  \]
  Note that the restriction of $\mathfrak{s'}$ to a tubular neighbourhood of $C$ is untwisted; therefore, it canonically induces a usual $\Spin^c$-structure on the neighbourhood.
  We define a ``$\mu$-map'' associated with $\mathfrak{s'}$ by
  \[
    \mu_{\mathcal{E}} \colon H_1(X'; \Z_2) \to H^1(\mathcal{B}^*; \Z_2), \quad \alpha \mapsto w_2(\mathcal{E}) / \alpha,
  \]
  where $\mathcal{E}$ is the universal characteristic $\mathrm{O}(2)$-bundle on $X' \times \mathcal{B}^*$.
  \begin{proposition} \label{prop: pin2 os formula}
    \[ 
      \SW^{\Pin^-(2)}(X', \mathfrak{s'})(\xi \cdot \mu_{\mathcal{E}}(C)) = \SW^{\Pin^-(2)}(X,\mathfrak{s})(\xi)
    \]
   for any $\xi \in H^*(\mathcal{B}^*; \Z_2)$.
  \end{proposition}
  \begin{proof}
    Fix a cylindrical-end metric on $X \setminus S$ modeld on the standard product metric on $[0,\infty) \times S^1 \times S^2$.
    This metric on $X \setminus S$ can be extended over both $S^1 \times D^3$ and $D^2 \times S^2$ to give metrics with non-negative scalar curvature.
    As noted above, the $\Spin^{c-}$-strucure $\mathfrak{s}$ induces a usual $\Spin^c$-structure on a neighbourhood of $S$, and so does $\mathfrak{s'}$ on a neighbourhood of $C$.
    Thus, the moduli spaces of solution of the $\Pin^-(2)$-monopole equations over $S^1 \times S^2$, $S^1 \times D^3$, and $D^2 \times S^2$ can be identified with the moduli spaces of reducible solutions of the ordinary Seiberg-Witten equations.
    We also observe that each solution of the $\Pin^-(2)$-monopole equations on $X$ and $X'$ restricts to a solution of the ordinary Seiberg-Witten equations near $S$ and $C$ respectively.
    The rest of the proof runs parallel to that of \cite{MR1863729}*{Proposition 2.2}.
   \end{proof}

  We next describe a surgery formula for the $\Z$-valued $\Pin^-(2)$-monopole invariant, which will be used to prove Theorem \ref{thm:Enriques}.
  Assume that the moduli space on $(X,\mathfrak{s})$ is orientable.
  As noted above, the restriction of $\mathfrak{s}$ and that of $\mathfrak{s'}$ canonically induce $\Spin^c$-structures on tubular neighbourhoods of $S$ and $C$ respectively.
  For a $\Spin^c$-structure, the determinant line bundle of its Dirac operators is always trivial.
  Then, by the excision property for the indices of families \citelist{\cite{MR1079726}*{7.1.3} \cite{N2}*{Lemma 6.10}}, we can show that the moduli space on $(X', \mathfrak{s'})$ is also orientable.
  Consequently, if the $\Z$-valued invariant $\SW^{\Pin^-(2)}_{\Z}(X,\mathfrak{s})$ is defined, so does $\SW^{\Pin^-(2)}_{\Z}(X', \mathfrak{s'})$.
  We define another ``$\mu$-map'' associated with $\mathfrak{s'}$ by
  \[
    \hat{\mu}_{\mathcal{E}} \colon H_1(X'; \ell') \to H^1(\mathcal{B}^*; \Z), \quad \alpha \mapsto \widetilde{c}_1(\mathcal{E}) / \alpha,
  \]
  where $\ell' := \widetilde{X'} \times_{\pm 1} \Z$.
  The proof of the following surgery formula also runs parallel to that of \cite{MR1863729}*{Proposition 2.2}.  
  \begin{proposition} \label{prop: Z surgery formula}
    Assume that the moduli space on $(X,\mathfrak{s})$ is orientable.
    We have, for any $\xi \in H^*(\mathcal{B}^*; \Z)$,
    \[
      \SW^{\Pin^-(2)}_{\Z}(X', \mathfrak{s'})(\xi \cdot \hat{\mu}_{\mathcal{E}}(C)) = \SW^{\Pin^-(2)}_{\Z}(X,\mathfrak{s})(\xi)
    \]
    for some orientations on the moduli spaces.
  \end{proposition}
  
\subsection{Irreducible $\Pin^-(2)$ and reducible $\Pin^-(2)$}
  We can establish another non-vanishing result based on a generalised blow-up formula for the $\Pin^-(2)$-monopole invariant \cite{N2}*{Theorem 3.9} and a gluing formula for irreducible $\Pin^-(2)$-monopoles and reducible $\Pin^-(2)$-monopoles \cite{N2}*{Theorem 3.11}
  It will play a key role in the proof of Theorem \ref{thm:Enriques0}.
    
  \begin{proposition} \label{prop:pin2-class b}
    Let $M$ be a closed, oriented, connected $4$-manifold that satisfies the following:
    \begin{itemize}
      \item there exists a non-trivial double cover $\widetilde{M} \to M$ with $b_+^{\ell_M}(M) \geq 2$, where $\ell_M=\widetilde{M} \times_{\{\pm 1\}} \Z$, and
      \item there exists a $\Spin^{c-}$-structure $\mathfrak{s}_M$ on $\widetilde{M}\to M$ such that $\widetilde{c}_1(\mathfrak{s}_M)^2 = 2\chi(M)+3\tau(M)$ and its $\Z_2$-valued $\Pin^-(2)$-monopole invariant is non-trivial.
    \end{itemize}
    Let $N$ be a closed, oriented, connected $4$-manifold with $b_+(N) = 0$.
    Let $Z$ be a connected sum of arbitrary positive number of $4$-manifolds, each of which belongs to one of the following types:
    \begin{enumerate}
      \item $S^2\times\Sigma$, where $\Sigma$ is a compact Riemann surface with positive genus,
      \item $S^1\times Y$, where $Y$ is a closed oriented $3$-manifold.
    \end{enumerate}
    Set $X := M \# N \# Z$.
    Then, there exist a non-trivial double cover $\widetilde{X} \to X$ and a $\Pin^-(2)$-basic class $\mathfrak{a} \in H^2(X; \ell_X)$, where $\ell_X := \widetilde{X} \times_{\{\pm 1\}} \Z$, such that
    \[
      (\mathfrak{a}^{+_g})^2 \ge 2\chi(M) + 3\tau(M)
    \]
    for any Riemannian metric on $X$.
  \end{proposition}
  
  \begin{proof}
    Set $X_1 := M \# N$ and $X_2 := Z$.
    We will first apply \cite{N2}*{Theorem 3.9} to $X_1 = M \# N$, and next \cite{N2}*{Theorem 3.11} to $X = X_1 \# X_2$ as follows.
    
    We can choose a set of non-trivial smooth loops $\gamma_1, \dots, \gamma_b$ in $N$ so that surgery along them produces a $4$-manifold $N'$ with $b_1(N')=0$ and $b_+(N')=0$.
    Conversely, we can find a set of homologically trivial embedded $2$-spheres in $N'$ so that surgery along them recovers $N$.
    We will identify $H^2(N;\Z)$ with $H^2(N';\Z)$.
    
    Set $X_1' := M \# N'$.
    Let $e_1, \dots, e_k$ be a set of generators for $H^2(N';\Z) / \mathrm{Tor}$ relative to which the intersection form is diagonal.
    By a generalised blow-up formula for the $\Pin^-(2)$-monopole invariant \cite{N2}*{Theorem 3.9}, we have a double covering $\widetilde{X_1'} \to X_1'$ and a unique $\Spin^{c-}$-structure $\mathfrak{s}_1'$ on it such that
    \[
      \widetilde{c}_1(\mathfrak{s}_1') = \widetilde{c}_1(\mathfrak{s}_M) + (\pm e_1 + \dots + \pm  e_k), 
    \]
    its $\Pin^-(2)$-monopole moduli space is $0$-dimensional, and its $\Pin^-(2)$-monopole invariant is equal to that of $(M,\mathfrak{s}_M)$.
    Here, the signs of $\pm e_i$ are arbitrary and independent of one another.
    
    By Proposition \ref{prop: pin2 os formula}, we have a double covering $\widetilde{X_1} \to X_1$ and a unique $\Spin^{c-}$-structure $\mathfrak{s}_1$ on it such that $\widetilde{c}_1(\mathfrak{s}_1) = \widetilde{c}_1(\mathfrak{s}_1')$ and
    \[
      \SW^{\Pin^-(2)}(X_1, \mathfrak{s}_1) (\mu_{\mathcal{E}}(\gamma_1) \cdots \mu_{\mathcal{E}}(\gamma_b)) = \SW^{\Pin^-(2)}(X_1', \mathfrak{s}_1')(1).
    \]
    
    We take a non-trivial double cover $\widetilde{X_2} \to X_2$ as described in Example \ref{ex: Z}, and choose any $\Spin^{c-}$-structure $\mathfrak{s}_2$ on $\widetilde{X_2} \to X_2$.
    Note that $\widetilde{c}_1(\mathfrak{s}_2)^2=0$.
  
    It now follows from \cite{N2}*{Theorem 3.11} that
    \[
      \mathfrak{a} := c_1(\mathfrak{s}_M) + (\pm e_1 + \dots + \pm e_k) + \widetilde{c}_1(\mathfrak{s}_2)
    \]
    is a $\Pin^-(2)$-basic class.
    Given a Riemannian metric $g$ on $X$, we can choose the signs of $\pm e_i$ so that
    \[
      (\mathfrak{a}^{+_g})^2 \ge (c_1(\mathfrak{s}_M) + \widetilde{c}_1(\mathfrak{s}_2))^2 = 2\chi(M) + 3\tau(M)
    \]
    holds \cite{MR2032500}*{Corollary 11}.
    This completes the proof.
  \end{proof}

\section{Computations of the Yamabe invariant} \label{section: Yamabe}
  Let us recall that we have
  \[
    \infs(X) := \inf_g \int_X |s_g|^2 \,\dv_g = 
    \begin{cases}
      (\yinv(X))^2 &\text{ if } \yinv(X) \le 0 \\
      0 &\text{ if } \yinv(X) \ge 0
     \end{cases}
  \]
  for any closed oriented $4$-manifold $X$ \cites{MR1085114, MR1674105}.
  
  \begin{proposition} \label{prop: upper bound}
    Let $M$, $N$, and $Z$ satisfy the assumptions in Proposition \ref{prop:pin2-class a} or Proposition \ref{prop:pin2-class b}.
    Set $X = M \# N \# Z$.
    Then, we have
    \[
      \infs(X) \ge 32\pi^2 \big( 2\chi(M) + 3\tau(M) \big).
    \]
  \end{proposition}
  \begin{proof}
    Proposition \ref{prop:pin2-class a} or Proposition \ref{prop:pin2-class b} and LeBrun's curvature estimate (\ref{eq:scalar}) imply that
    \[
      \int_X s_g^2 \,\dv_g \ge 32 \pi^2 (\mathfrak{a}^{+_g})^2 \ge 32 \pi^2 \big( 2\chi(M) + 3\tau(M) \big)
    \]
    for any Riemannian metric $g$ on $X$.
  \end{proof}
  
  \begin{proof}[Proof of Theorem \ref{thm:yamabe} and Theorem \ref{thm:Enriques0}]
    Let $M$, $N$, and $Z$ satisfy the assumptions in Theorem \ref{thm:yamabe} or Theorem \ref{thm:Enriques0}.
    We have $\infs(M) = 32\pi^2 c_1^2(M)$ by \cites{MR1386835, MR1674105}, and $\infs(N) = \infs(Z) = 0$ by assumption.
    We remark that an Enriques surface satisfies the assumption for $M$ in Proposition \ref{prop:pin2-class b} by \cite{N2}*{Theorem 1.3}.
    Set $X := M \# N \# Z$.
    
    By Proposition \ref{prop: upper bound}, we have
    \[
      \infs(X) \ge 32 \pi^2 \big( 2\chi(M) + 3\tau(M) \big) = 32 \pi^2 c_1^2(M).
    \]
    On the other hand, by \cite{MR2032500}*{Proposition 13}, we have
    \[
      \infs(X) \le \infs(M) + \infs(N) + \infs(Z) = 32\pi^2 c_1^2(M).
    \]
    Since $X$ has a $\Pin^-(2)$-basic class, $\yinv(X) \le 0$.
    Thus,
    \[
      \yinv(X) = - \sqrt{\infs(X)} = - 4\pi \sqrt{2c_1^2(M)}.
    \]
    This completes the proof.
  \end{proof}
   
  \begin{proof}[Proof of Theorem \ref{thm:Enriques}]
    Let $M$ be an Enriques surface.
    By \cite{N2}*{Theorem 1.13}, the $\Z$-valued $\Pin^-(2)$-monopole invariant of $mM$ is non-trivial for any $m \ge 2$.
    
    If $b_1(N) = 0$, by \cite{N2}*{Theorem 3.9}, the $\Z$-valued $\Pin^-(2)$-monopole invariant of $mM \# N$ is non-trivial.
    We remark that \cite{N2}*{Theorem 3.9} holds for the $\Z$-valued $\Pin^-(2)$-monopole invariant.
    If $b_1(N) > 0$, by using Proposition \ref{prop: Z surgery formula} as in the proof of Proposition \ref{prop:pin2-class a} or that of Proposition \ref{prop:pin2-class b}, we are reduced to the case when $b_1(N) = 0$.
    Thus, the $\Z$-valued $\Pin^-(2)$-monopole invariant of $mM \# N$ is non-trivial.
    In particular, we have
    \[
      \yinv(mM \# N) \le 0.
    \]
    On the other hand, we have
    \[
      0 \le \infs(mM \# N) \le m \infs(M) + \infs(N) = 0.
    \]
    Thus, $\yinv(mM \# N)=0$.
    
    Since $2\chi(mM \# N) + 3\tau(mM \# N)< 0$, by the Hitchin-Thorpe inequality, it does not admit Ricci-flat metrics.
    Consequently, it does not admit Riemannian metrics of non-negative scalar curvature.
  \end{proof}
  
\section{Obstructions to Einstein metrics}

  We begin by examining LeBrun's inequalities (Cf. \cite{MR1872548}*{Proposition 3.2}).
  
  \begin{lemma} \label{lemma: equality}
    Let $M$, $N$, and $Z$ satisfy the assumptions in Proposition \ref{prop:pin2-class a} or Proposition \ref{prop:pin2-class b}.
    If equality holds in either (\ref{eq:scalar}) or (\ref{eq:weyl}) for some Riemannian metric $g$ on $X := M \# N \# Z$, then $\mathfrak{a}^{+_g} = 0$.
  \end{lemma}
  \begin{proof}
    Suppose that equality holds and $\mathfrak{a}^{+_g} \ne 0$.
    Proposition \ref{prop:scalar} implies that the double cover $\widetilde{X} = \widetilde{M \# Z} $ admits an almost-K\"{a}hler structure; therefore, its ordinary Seiberg-Witten invariant is non-trivial \cite{MR1306023}.
    On the other hand,  $\widetilde{X} = M \# M \# N \# N \# \widetilde{Z}$ or $\widetilde{X} = \widetilde{M} \# N \# N \# \widetilde{Z} \# (S^1\times S^3)$ according as $M$ satisfies the assumptions of Proposition \ref{prop:pin2-class a} or those of Proposition \ref{prop:pin2-class b}; in either case, $\widetilde{X}$ has at least two connected-summands with positive $b_+$; thus, its ordinary Seiberg-Witten invariant is trivial.
    This is a contradiction.
  \end{proof}
  
  \begin{proposition} \label{prop: strict ineq}
    Let $M$, $N$, and $Z$ satisfy the assumptions in Proposition \ref{prop:pin2-class a} or Proposition \ref{prop:pin2-class b}.
    Then, we have a strict inequality
    \[
      \frac{1}{4\pi^2} \int_X \left(\frac{s_g^2}{24} + 2|W^+_g|^2\right) \,\dv_g > \frac{2}{3} \big( 2\chi(M) + 3\tau(M) \big)
    \]
    for any Riemannian metric $g$ on $X := M \# N \# Z$.
  \end{proposition}
  
  \begin{proof}
    Combined with Proposition \ref{prop:pin2-class a} or Proposition \ref{prop:pin2-class b}, the Cauchy-Schwarz inequality and LeBrun's curvature estimate (\ref{eq:weyl}) yield
    \begin{equation*}
      \begin{aligned}
      \frac{1}{4\pi^2} \int_X \left(\frac{s_g^2}{24} + 2|W^+_g|^2\right) \,\dv_g &\ge \frac{1}{4\pi^2} \frac{1}{27} \int_X \big(s_g - \sqrt{6}|W^+_g| \big)^2 \,\dv_g \\
      &\ge \frac{1}{4\pi^2} \frac{1}{27} \cdot 72\pi^2 (\mathfrak{a}^{
      +_g})^2 \\
      &\ge \frac{2}{3} \big(2\chi(M) + 3\tau(M)\big)
      \end{aligned}
    \end{equation*}
    for any Riemannian metric $g$ on $X$ (Cf. \cite{MR1872548}*{Proposition 3.1}).
    
    We remark that $X$ is not diffeomorphic to a finite quotient of a K3 surface or $T^4$; in particular, $X$ does not admit a Ricci-flat anti-self-dual metric \cite{MR0350657}.
    Suppose that equality holds for some Riemannian metric $g$ on $X$.
    By Lemma \ref{lemma: equality}, we have $\mathfrak{a}^{+_g} = 0$; therefore, $s_g = W^+_g = 0$.
    Note that $X$ does not admit a Riemannian metric of positive scalar curvature by Proposition \ref{prop:pin2-class a} or Proposition \ref{prop:pin2-class b}.
    Consequently, $g$ is Ricci-flat and anti-self-dual.
    This is a contradiction.
  \end{proof}

  Proposition \ref{prop: strict ineq} leads to a new obstruction to the existence of Einstein metrics (Cf. \cite{MR2032500}*{Section 6}).
  
  \begin{theorem}
    Let $M$, $N$, and $Z$ satisfy the assumptions in Proposition \ref{prop:pin2-class a} or Proposition \ref{prop:pin2-class b}.
    If $X := M \# N \# Z$ admits an Einstein metric, then
    \[
      \frac{1}{3} \big( 2\chi(M) + 3\tau(M) \big) > 4 - \big( 2\chi(N \# Z) + 3\tau(N \# Z) \big).
    \]
  \end{theorem}
  
  \begin{proof}
    We first note that
    \begin{equation*}
    \begin{split}
        2\chi(X) + 3\tau(X) &=  2\big( \chi(M) + \chi(N \#Z ) -2 \big) + 3\big( \tau(M) + \tau(N \# Z) \big) \\
        &= 2\chi(M) + 3\tau(M) + 2\chi(N \# Z) + 3\tau(N \# Z) - 4.
    \end{split}
    \end{equation*}
    By the Chern-Gauss-Bonnet formula and the Hirzebruch signature theorem, if $X$ admits an Einstein metric, we have
    \[
      2\chi(X) + 3\tau(X) = \frac{1}{4\pi^2} \int_X \left(\frac{s_g^2}{24} + 2|W^+_g|^2\right) \,\dv_g.
    \]
    By Proposition \ref{prop: strict ineq}, we have a strict inequality
    \[
      \frac{1}{4\pi^2} \int_X \left(\frac{s_g^2}{24} + 2|W^+_g|^2\right) \,\dv_g > \frac{2}{3} \big( 2\chi(M) + 3\tau(M) \big).
    \]
    Thus, we have
    \[
      2\chi(M) + 3\tau(M) + 2\chi(N \# Z) + 3\tau(N \# Z) - 4 > \frac{2}{3} \big( 2\chi(M) + 3\tau(M) \big).
    \]
    The proof is completed by rearranging terms.
  \end{proof}

  \begin{theorem}
    Let $M$, $N$, and $Z$ satisfy the assumptions in Proposition \ref{prop:pin2-class a} or Proposition \ref{prop:pin2-class b}.
    If $X := M \# N \# Z$ admits an anti-self-dual Einstein metric, then
    \[
      \frac{1}{4} \big( 2\chi(M) + 3\tau(M) \big) > 4 - \big( 2\chi(N \# Z) + 3\tau(N \# Z) \big).
    \]    
  \end{theorem}
  
  \begin{proof}
    We first note that
    \begin{equation*}
    \begin{split}
        2\chi(X) + 3\tau(X) &=  2\big( \chi(M) + \chi(N \#Z ) -2 \big) + 3\big( \tau(M) + \tau(N \# Z) \big) \\
        &= 2\chi(M) + 3\tau(M) + 2\chi(N \# Z) + 3\tau(N \# Z) - 4.
    \end{split}
    \end{equation*}
    By the Chern-Gauss-Bonnet formula and the Hirzebruch signature theorem, if $X$ admits an anti-self-dual Einstein metric, we have
    \[
      2\chi(X) + 3\tau(X) = \frac{1}{4\pi^2}\int_X \frac{s_g^2}{24} \,\dv_g = \frac{1}{96\pi^2}\int_X s_g^2 \,\dv_g.
    \]
    We have a strict inequality
    \[
      \int_X s_g^2 \,\dv_g > 72\pi^2  \big( 2\chi(M) + 3\tau(M) \big) = 96\pi^2 \cdot \frac{3}{4} \big( 2\chi(M) + 3\tau(M) \big),
    \]
    which follows by the same method as in Proposition \ref{prop: strict ineq} using LeBrun's curvature estimate (\ref{eq:weyl}) and Lemma \ref{lemma: equality}.
    Thus, we have
    \[
      2\chi(M) + 3\tau(M) + 2\chi(N \# Z) + 3\tau(N \# Z) - 4 > \frac{3}{4} \big( 2\chi(M) + 3\tau(M) \big).
    \]
    The proof is completed by rearranging terms.
  \end{proof}

  \begin{example}
    Mumford constructed a compact complex surface $K$ of general type that is homeomorphic to the complex projective plane \cite{MR527834}.
    Let $M$ be a closed symplectic manifold with $b_+(M) \ge 2$.
    Let $Z$ be a connected sum of arbitrary positive number of $4$-manifolds, each of which belongs to one of the following types:
    \begin{enumerate}
      \item $S^2\times\Sigma$, where $\Sigma$ is a compact Riemann surface with positive genus,
      \item $S^1\times Y$, where $Y$ is a closed oriented $3$-manifold.
    \end{enumerate}
    Then, $M \# m\overline{\CP^2} \# n\overline{K} \# Z$ does not admit an Einstein metric if
    \[
      4 - 5(n+m)  \ge \big( 2\chi(Z) + 3\tau(Z) \big) + \frac{1}{3} \big( 2\chi(M) + 3\tau(M) \big),
    \]
    and it does not admit an anti-self-dual Einstein metric if
    \[
      4 - 5(n+m)  \ge \big( 2\chi(Z) + 3\tau(Z) \big) + \frac{1}{4} \big( 2\chi(M) + 3\tau(M) \big).
    \]
  \end{example}
  
  We end this section by examining an equality related to Proposition \ref{prop: strict ineq}, the proof of which is worth mentioning here although it will not play any role in our work.
  
  \begin{proposition}
    Let $\pi \colon \widetilde{M} \to M$ satisfy the the assumptions in Proposition \ref{prop:pin2-class b}.
    If there exists a Riemannian metric $g$ on $M$ that satisfies
    \[
      \frac{1}{4\pi^2} \int_M \left(\frac{s_g^2}{24} + 2|W^+_g|^2\right) \,\dv_g = \frac{2}{3} \big( 2\chi(M) + 3\tau(M) \big),
    \]
    then $(\widetilde{M}, \pi^* g)$ is a K3 surface or $T^4$ with hyper-K\"ahler metric and the covering transformation of $\widetilde{M}$ is anti-holomorphic; moreover, $M$ is an Enriques surface if $\widetilde{M}$ is a K3 surface.
  \end{proposition}
  
  \begin{proof}
     It follows from a similar argument as in \cite{MR1872548}*{Proposition 3.2} that  $(\widetilde{M}, \pi^*g)$ is a K3 surface or $T^4$ with hyperK\"ahler metric, and that the covering transformation is anti-holomorphic.
     By ``Donaldson's trick'' \citelist{\cite{MR1171888} \cite{MR1795406}*{Section 15.1}}, we can show that there exists another complex structure on $\widetilde{M}$ compatible with $\pi^* g$ for which the covering transformation is holomorphic; in particular, $M$ is an Enriques surface if $\widetilde{M}$ is a K3 surface.
  \end{proof}
  
\section{Obstructions to long-time Ricci flows}

  Recall that a long-time solution of the normalised Ricci flow is a family of Riemannian metrics that satisfies
  \[
    \frac{\partial}{\partial t} g(t) = -2 \Ric_{g(t)} + \frac{2}{m} \left( \frac{\int_X s_{g(t)} \,\dv_{g(t)}}{\int_X \dv_{g(t)}} \right) g(t)
  \]
  for $t \in [0,\infty)$.
  Proposition \ref{prop: strict ineq} also leads to a new obstruction to the existence of long-time solutions of the normalised Ricci flow with uniformly bounded scalar curvature (Cf. \cite{ishida}*{Section 5}).
  
  \begin{lemma} \label{lemma: yamabe negative}
    Let $M$, $N$, and $Z$ satisfy the assumptions in Proposition \ref{prop:pin2-class a} or Proposition \ref{prop:pin2-class b}.
    If $X := M \# N \# Z$ admits a long-time solution of the normalised Ricci flow with uniformly bounded scalar curvature, then we have $\yinv(X) < 0$.
  \end{lemma}
  
  \begin{proof}
    By Proposition \ref{prop:pin2-class a} or Proposition \ref{prop:pin2-class b}, $X$ has a $\Pin^-(2)$-basic class; hence, $\yinv(X) \le 0$.
    Then, by \cite{MR2289603}*{Theorem A} and \cite{MR2869163}*{Theorem 1.1}, we have a Hitchin-Thorpe type inequality
    \[
      2\chi(X) - 3|\tau(X)| \ge \frac{1}{96\pi^2} \yinv(X)^2.
    \]
    Thus, $2\chi(X) + 3\tau(X) \ge 0$.
    Note that $2\chi(N \# Z) + 3\tau(N \# Z) < 0$.
    Thus, we get
    \[
      2\chi(M) + 3\tau(M) > 0.
    \]
    By Proposition \ref{prop: upper bound}, we have
    \[
      \yinv(X) = - \infs(X) \le - 32\pi^2 \big( 2\chi(M) + 3\tau(M) \big) < 0.
    \]
    This completes the proof.
  \end{proof}
  
  \begin{theorem}
    Let $M$, $N$, and $Z$ satisfy the assumptions in Proposition \ref{prop:pin2-class a} or Proposition \ref{prop:pin2-class b}.
    If $X := M \# N \# Z$ admits a long-time solution of the normalised Ricci flow with uniformly bounded scalar curvature, then
    \[
      4 - \big( 2\chi(N \# Z) + 3\tau(N \# Z) \big) \le \frac{1}{3} \big( 2\chi(M) + 3\tau(M) \big).
    \]
  \end{theorem}
  
  \begin{proof}
    By Lemma \ref{lemma: yamabe negative}, we have $\yinv(X) < 0$.
    Then, by \cite{ishida}*{Proposition 5}, we have
    \[
      \sup_{t \in [0,\infty)} \min_{x \in X} s_{g(t)}(x) < 0.
    \]
    Thus, by \cite{MR2357999}*{Lemma 3.1}, we have
    \[
      \int_0^{\infty} \int_X \big|\tlRicci_{g(t)}\big|^2 \,d\mu_{g(t)} \,dt < \infty,
    \]
    where we denote by $\tlRicci$ the traceless Ricci tensor.
    Hence, we have
    \begin{equation} \label{eq: tlricci}
      \lim_{m \to \infty} \int_m^{m+1} \int_X \big|\tlRicci_{g(t)}\big|^2 \,d\mu_{g(t)} \,dt = 0.
    \end{equation}
    By the Chern-Gauss-Bonnet formula and the Hirzebruch signature theorem, we have
    \[
      2\chi(X) + 3\tau(X) = \frac{1}{4\pi^2} \int_X \bigg( \frac{s_{g(t)}^2}{24} + 2|W^+_{g(t)}|^2 - \frac{\big|\tlRicci_{g(t)}\big|^2}{2} \bigg) \,d\mu_{g(t)}
    \]
    for any $t \in [0,\infty)$.
    Hence, we have
    \begin{equation} \label{eq: gauss-bonnet}
      2\chi(X) + 3\tau(X) = \frac{1}{4\pi^2} \int_m^{m+1} \int_X \bigg( \frac{s_{g(t)}^2}{24} + 2|W^+_{g(t)}|^2 - \frac{\big|\tlRicci_{g(t)}\big|^2}{2} \bigg) \,d\mu_{g(t)} \,dt
    \end{equation}
    for any $m \in [0,\infty)$.
    By (\ref{eq: tlricci}) and (\ref{eq: gauss-bonnet}), we have
    \[
      2\chi(X) + 3\tau(X) = \lim_{m \to \infty} \frac{1}{4\pi^2} \int_m^{m+1} \int_X \bigg( \frac{s_{g(t)}^2}{24} + 2|W^+_{g(t)}|^2 \bigg) \,d\mu_{g(t)} \,dt.
    \]
    On the other hand, by Lemma \ref{prop: strict ineq}, we have
    \[
      \frac{1}{4\pi^2} \int_X \bigg( \frac{s_{g(t)}^2}{24} + 2|W^+_{g(t)}|^2 \bigg) \,d\mu_{g(t)} > \frac{2}{3} \big( 2\chi(M) + 3\tau(M) \big)
    \]
    for any $t \in [0,\infty)$.
    Thus,
    \[
      \lim_{m \to \infty} \frac{1}{4\pi^2} \int_m^{m+1} \int_X \bigg( \frac{s_{g(t)}^2}{24} + 2|W^+_{g(t)}|^2 \bigg) \,d\mu_{g(t)} \,dt \ge \frac{2}{3} \big( 2\chi(M) + 3\tau(M) \big).
    \]
    Consequently, we get
    \[
      2\chi(X) + 3\tau(X) \ge \frac{2}{3} \big( 2\chi(M) + 3\tau(M) \big).
    \]
    This completes the proof.
  \end{proof}

\end{document}